\documentclass{amsart}

\usepackage{graphicx}

\newtheorem{theorem}{Theorem}[section]
\newtheorem*{RC}{Rademacher's Conjecture}
\newtheorem{conj}[theorem]{Conjecture}

\theoremstyle{remark}
\newtheorem{remark}[theorem]{Remark}

\numberwithin{equation}{section}

\begin{document}

\title[ {Rademacher's Conjecture is} {\tiny (\MakeLowercase{almost certainly})} False]{{\large Rademacher's Infinite Partial Fraction Conjecture
is} {\tiny (\MakeLowercase{almost certainly})} {\large False} }

\author[A. V. Sills]{Andrew V. Sills}
\address{Andrew V. Sills, Department of Mathematical Sciences, Georgia Southern University, 65 Georgia Avenue, Room 3008, Statesboro, Georgia 30458-8093, USA}
\email{ASills@GeorgiaSouthern.edu}
\thanks{A. V. S. thanks the Center for Discrete Mathematics and Theoretical Computer Science (DIMACS) for its hospitality during
his July 2011 stay, which led to this collaboration with D. Z.}

\author[D. Zeilberger]{Doron Zeilberger}
\address{Doron Zeilberger, Department of Mathematics, Rutgers University, Hill Center, 110 Frelinghuysen Road, Piscataway, NJ 08854-8019, USA}
\email{zeilberg@math.rutgers.edu}
\thanks{D. Z. is supported in part by a grant of the National Science Foundation}


\date{\today}

\dedicatory{``It is a capital mistake to theorise before one has data.  Insensibly one begins to twist facts to suit theories, instead of theories to suit facts."---Sherlock Holmes to Dr. Watson~\cite[p. 63]{D}.}

\begin{abstract}
 In his book \emph{Topics in Analytic Number Theory}, Hans Rademacher conjectured
that the limits of certain sequences of coefficients that arise in the ordinary partial fraction decomposition
of the generating function for partitions of integers into at most $N$ parts exist and equal
particular values that he specified.  Despite being open for nearly four decades, little progress
has been made toward proving or disproving the conjecture, perhaps in part due to the
difficulty in actually computing the coefficients in question.  
  In this paper, we provide a fast algorithm for calculating the Rademacher coefficients,
a large amount of data, 
direct formulas for certain collections of Rademacher coefficients, and overwhelming evidence
against the truth of the conjecture.  While the limits of the sequences of Rademacher coefficients
do not exist (the sequences oscillate and attain arbitrarily large positive and negative values),
the sequences do get very close to Rademacher's conjectured limits for certain (predictable)
indices in the sequences.
\end{abstract}

\maketitle

\section*{Important Note} This article is accompanied by the Maple package {\tt HANS}, downloadable from \\
\texttt{http://www.math.rutgers.edu/\~{}zeilberg/tokhniot/HANS}  . \\ 

The ``front'' of this article, \\
\texttt{http://www.math.rutgers.edu/\~{}zeilberg/mamarim/mamarimhtml/hans.html}\\
contains lots of supporting input and output files.

\section{Introduction}
Let $p_N(n)$ denote the number of partitions of the integer $n$ into at most $N$ parts.
The generating function of $p_N(n),$
\[ F_N(x):= \sum_{n\geq 0} p_N(n) x^n = \prod_{j=1}^N \frac{1}{1-x^j}, \]
may be decomposed into partial fractions:
\begin{equation} \label{pfd}
 \prod_{j=1}^N \frac{1}{1-x^j} = \sum_{k=1}^N \underset{\gcd(h,k)=1} {\sum_{0\leq h < k } }\sum_{l=1}^{\lfloor N/k \rfloor }
 \frac{C_{h,k,l}(N) }{ (x - e^{2\pi i h/k } )^l}.\end{equation}
We shall refer to the $C_{h,k,l}(N)$ defined by~\eqref{pfd} as the \emph{Rademacher coefficients}.

Near the end of his posthumously published masterpiece \emph{Topics in Analytic
Number Theory}~\cite[p. 302]{R}, Hans Rademacher made the following conjecture:
 
\begin{RC}
For all integers $h,k,l$ such that $0 \leq h<k$, $\gcd(h,k)=1$ and $l \geq 1$,  $\lim_{N\to\infty} C_{h,k,l} (N)$ exists and equals
 \begin{equation}
R_{h,k,l}:= -2\pi \left( \frac{\pi}{12} \right)^{3/2} \frac{ e^{\pi i (s(h,k)+2hl/k)}}{k^{5/2}}
  \Delta^{l-1}_{\alpha} L_{3/2}\left( -\frac{\pi^2}{6k^2} (\alpha+1) \right) , \label{RadConjValue}
\end{equation}
evaluated at $\alpha=\frac{1}{24}$,  where
$s(h,k) = \sum_{\mu=1}^{k-1}\left( \frac{\mu}{k} - \lfloor \frac{\mu}{k} \rfloor - \frac 12\right)
\left( \frac{h\mu}{k} - \lfloor \frac{h\mu}{k} \rfloor - \frac 12\right)$ is the Dedekind sum, 
$\Delta_\alpha$ is the forward difference
operator, so that 
\[ \Delta_\alpha^j f(\alpha) = \sum_{h=0}^j (-1)^h \binom{j}{h} f(\alpha+j-h), \]
and
\[ L_{3/2} (-y^2) = -\frac{1}{2\sqrt{\pi} y^2} \left( 2 \cos (2y) - \frac{\sin (2y)}{y} \right). \]
\end{RC}

If the Rademacher conjecture would have been true then it would have followed that
\[ \lim_{N\to\infty} C_{0,1,1} (N) =  R_{0,1,1} (=- \frac{6}{25}\left( 1 + \frac{2\sqrt{3}}{5\pi} \right) = -0.292927573960\dots), \]
\[ \lim_{N\to\infty} C_{0,1,2} (N) =  R_{0,1,2} (=\frac{144}{1225}+ \frac{5616}{42875\pi} =0.1897670688440\dots) , \]
\[ \lim_{N\to\infty} C_{1,2,1} (N) = R_{1,2,1} (=- \frac{2\sqrt{6}}{25}\left( \cos\frac{5\pi}{12} - \frac{12}{5\pi}\sin\frac{5\pi}{12} \right) = 0.093882853484\dots). \]

\begin{remark} 
The floating-point approximation for the value of $R_{0,1,1}$ stated by
Rademacher~\cite[p. 302]{R} was erroneous, as were the exact values of $R_{0,1,2}$ and $R_{1,2,1}$ (for the latter he gave exactly one half
of the correct value). These erroneous values  were quoted, without correction, by Andrews~\cite[p. 388]{A}.
\end{remark}

Rademacher supplied (with one error) a table of values for $C_{0,1,1}(N)$, $C_{0,1,2}(N)$, and $C_{1,2,1}(N)$ for
$N=1,2,3,4,5$, and in fact these values are not too far off from his conjectured ``$N=\infty$'' cases.

Rademacher began work on the book in which this conjecture appeared~\cite{R} no
later than 1944 and was still working on it at the inception of his final illness.  Thus as the final version was edited and published by Rademacher's students Emil Grosswald, Joseph Lehner, and Morris Newman,
after Rademacher's death in 1969, we will never know whether Rademacher came to doubt the truth of the conjecture after
he had written it down.  However, George Andrews reports that Rademacher discussed 
the conjecture in a course he taught at the University of Pennsylvania during the 1961--1962
academic year.

In~\cite{M}, Augustine Munagi considered a different type of partial fraction decomposition
called $q$-partial fractions, and proved a special case of the analog of the Rademacher
conjecture, relative to the $q$-partial fraction decomposition. 

We should also note that the first to cast doubts on the Rademacher conjecture were Jane Friedman and
Leon Ehrenpreis. Ehrenpreis~\cite[p. 317]{E} stated,  
``If one attempts to carry out the usual type of
partial fraction decomposition of the partition function term-by-term, it is difficult to compute the
coefficients.  
My student, Jane Friedman, spent a great deal of time trying to apply computer algorithm
methods to compare the coefficients with those of Rademacher\dots Unfortunately, the
computer study proved inconclusive.'' 

In this
article, we present overwhelming evidence against this conjecture, taken literally, but we will present ample
evidence for a modified conjecture.    Additionally, we will present a fast algorithm for 
generating the Rademacher coefficients, and formulas for a selection of particular 
Rademacher coefficients.

\section{Empirical evidence against the Rademacher conjecture}
\subsection{The \emph{actual} behavior of the sequences $C_{0,1,l}(N)$, $l = 1,2,3,\dots$}

 At\\
 \texttt{http://www.math.rutgers.edu/\~{}zeilberg/mamarim/mamarimhtml/hans.html}     ,\\
 there are links to various files including
 \begin{itemize}
   \item the sequences $C_{0,1,l}(N)$ for $1\leq l \leq 10$ and $1\leq N \leq 850$, in Maple-readable format,
  \item the sequences of floating point approximations to $C_{0,1,l}(N)$ for $1\leq l \leq 40$ and $1\leq N \leq 1000$, in Maple-readable format.
   \end{itemize}

Figures~\ref{C011-100} and~\ref{C011-200} are graphical summaries of $C_{0,1,1}(N)$.
 \begin{center}
  \begin{figure}[h] 
  \includegraphics[scale=1.00]{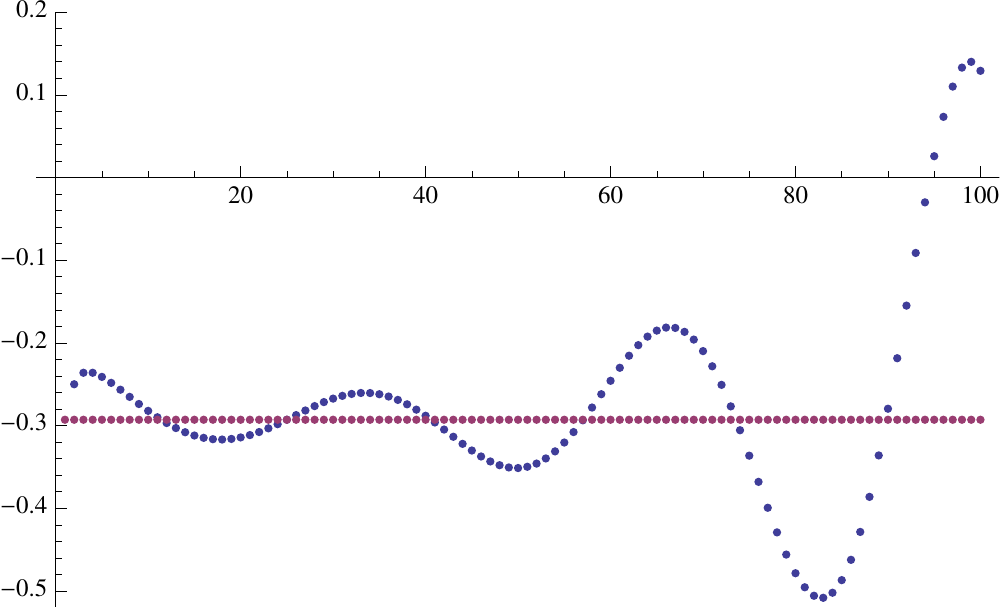}
 \caption{Graph of $C_{0,1,1}(N)$ for $N$ from $1$ to
 $100$, together with the line $y=R_{0,1,1}$ }
  \label{C011-100}
\end{figure}
 
 \begin{figure}[h] 
  \includegraphics[scale=1.00]{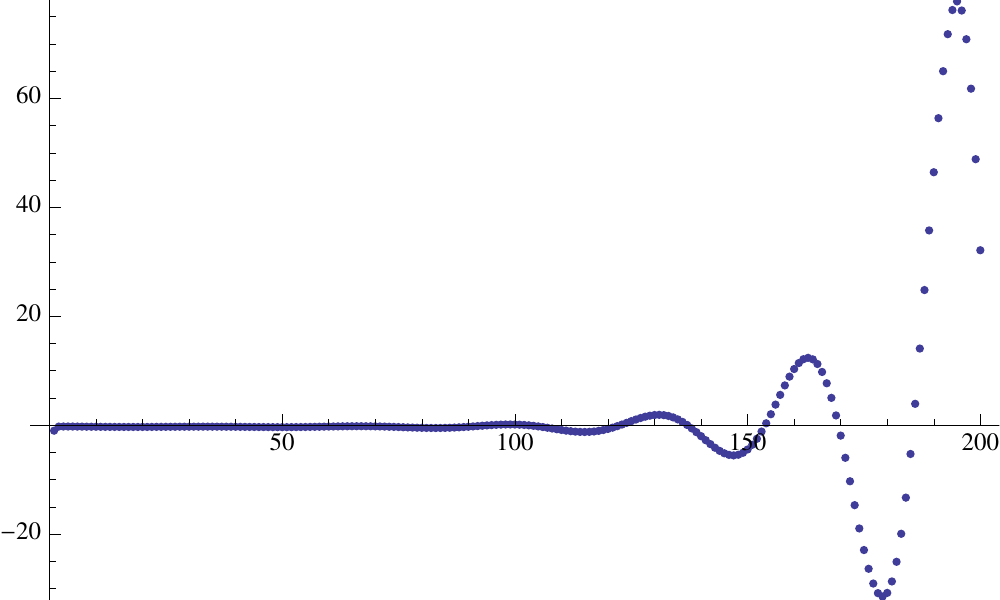}
\caption{Graph of $C_{0,1,1}(N)$ for $N$ from $1$ to
 $200$. }
 \label{C011-200}
\end{figure}
\end{center}

Note that $C_{0,1,1}(25)$ differs from $R_{0,1,1}$, Rademacher's conjectured value of
$C_{0,1,1}(\infty)$, by less than $0.000032$, but then things go down hill (for
the conjecture) from there, and get particularly bad after about $n=150$.
Numerical evidence points to the sequence $C_{0,1,1}(N)$ oscillating
and attaining arbitrarily large positive and negative values.
The same thing is true for other sequences $C_{h,k,l}(N)$; see the output file at \\
\texttt{http://www.math.rutgers.edu/\~{}zeilberg/tokhniot/oHANS11}        .
\\

By examining the graphs in Figures~\ref{C011-100} and~\ref{C011-200} and the associated numerical
data, it seems reasonable to
state the following alternative conjecture:
\begin{conj}
 $C_{0,1,1}(N)$ is an oscillating function of $N$ of ``period" $32$, with local maxima (resp. 
 local minima) that attain arbitrarily large positive (resp. negative) values as $N$ increases.
 \end{conj}

By ``period $32$" we mean that the peaks and valleys,  eventually, recur at a period of $32$. We also
noticed, numerically, that the elevations and depths of successive peaks and valleys roughly grows
exponentially with a factor around $8$.  Specifically, $C_{0,1,1}(N)$ has local maxima at
$N=3,4,33,66,99,131,163,195,227,259,291,323,$\\
$355,387,419, 451, 483, 515, 547, 579, 611, 643,
675, 707, 739,771, \dots$.  The ratio of consecutive local maxima is
\begin{multline*}
\{ 1, 1.103504574, 0.6965131681, -0.7709983810, 13.63072659, 6.485614677, \\6.289519948,
    6.547018652, 6.785098547, 6.992410281, 7.161220864, 7.301859590,\\
    7.420337150, 7.521483398, 7.608822684, 7.684977203, 7.751953124,\\
    7.811301903, 7.864245038, 7.911756412, 7.954622120, 7.993483579,\\
    8.028869316, 8.061218737, 8.090900135, \dots \} \end{multline*}
 $C_{0,1,1}(N)$ has a local minimum half way between
each local maximum, at 
$N=18,50,83,115,147,179,211,243,275,307,339,371,403,435,467,499,531,563,595,\\
627,659,691,723,755,787,\dots$.
\vskip 3mm

More precise conjectured asymptotics for $C_{0,1,l}(N)$ for $l$ between $1$ and $15$ can be gotten from

\texttt{http://www.math.rutgers.edu/\~{}zeilberg/tokhniot/oHANS10} \quad . \\

It appears that all $C_{0,1,l}(N)$ have a ``period'' of $32$. The locations of the
local maxima and minima of $C_{0,1,l}(N)$, from $N=99$ until $N=803$
occur when $N \equiv 3\pmod{32} $ and $N \equiv 19\pmod {32}$ respectively.  For $N>803$, they
become $N \equiv 2\pmod{32}$ and $N \equiv 18\pmod{32}$ respectively. This gives (very meager) evidence of a ``shifting of the perihelion'',
but one would need to go much further to investigate this.
More generally, for $C_{0,1,l}(N)$, and for $100 \leq N \leq 800$ these locations are congruent to $12-9l \pmod {32}$, 
and $28-9l \pmod{32}$
respectively. Probably these too would eventually get shifted (ever so slowly).

\subsection{How to compute the sequences $C_{0,1,l}(N)$ fast} 
If you use the {\it definition} of $C_{0,1,1}(N)$, or Andrews's formula~\cite[p. 388, Theorem 1]{A}, you can't go very far, even with Maple.
Rademacher calculated $C_{0,1,1}(N)$ for $N=1,2,3,4,5$, presumably by hand, and made
an error in the $N=5$ case.
Andrews~\cite[p. 388]{A}, who had access to a computer algebra system in 2003, 
corrected Rademacher's error at $N=5$ and extended the
list to $N=6,7,8$.

We need to be more clever. A fast recurrence for $C_{0,1,l}(N)$ can be derived as follows.
Since
$$
\prod_{j=1}^N \frac{1}{1-x^j}=\sum_{l=1}^{N} \frac{C_{0,1,l}(N)}{(x-1)^l} + \dots  \quad ,
$$
we can multiply both sides by $(x-1)^N$ and get
$$
(x-1)^N\prod_{j=1}^N \frac{1}{1-x^j}=\sum_{r=0}^{N-1} D_r(N){(x-1)^r} + \dots  \quad .
$$
Once we know $D_r(N)$, we can find $C_{0,1,l}(N)$, since they equal $D_{N-l}(N)$. It remains to
find a fast recurrence for $D_r(N)$. \\

By definition, we have:
$$
\frac{1-x^N}{x-1} \left ( \sum_{r=0}^{\infty} D_r(N){(x-1)^r} \right )=
 \sum_{r=0}^{\infty} D_r(N-1){(x-1)^r} \quad.
$$
Letting $z=x-1$, this is:
$$
\frac{1-(z+1)^N}{z} \left ( \sum_{r=0}^{\infty} D_r(N){z^r} \right )=
 \sum_{r=0}^{\infty} D_r(N-1){z^r} \quad.
$$
By the binomial theorem,
$$
- \left ( \sum_{a=0}^{N-1} \binom{N}{a+1} z^a \right ) \left ( \sum_{r=0}^{\infty} D_r(N){z^r} \right )=
 \sum_{r=0}^{\infty} D_r(N-1){z^r} \quad.
$$
Equating coefficients of $z^r$ we get:
$$
ND_r(N)+\sum_{a=1}^{r} \binom{N}{a+1} D_{r-a}(N)= -D_r(N-1) \quad .
$$
And finally:
$$
D_r(N)=-\frac{D_r(N-1)}{N}-\sum_{a=1}^{r} \frac{1}{N}\binom{N}{a+1} D_{r-a}(N) \quad .
$$
This is implemented in procedure {\tt C01(l,N)} of {\tt HANS}.

The same argument leads to efficient recurrences for $C_{h,k,l}(N)$, except that now we have
to distinguish between the case when $N$ is divisible by $k$ and when it is not, yielding
two different recurrences. This is implemented in procedure {\tt ChklN(h,k,l,N)} of {\tt HANS}. \\

\subsection{$C_{1,2,1}(N)$}

The values of $C_{1,2,1}(N)$ for $1\leq N\leq 700$, in both exact rational form and
approximate floating point form are provided at

\texttt{http://www.math.rutgers.edu/\~{}zeilberg/tokhniot/oHANS3}     .

Graphical summaries are provided in Figures~\ref{C121-150} and~\ref{C121-300}.
As the graphs make clear, it is better to regard $C_{1,2,1}(N)$ as two separate subsequences,
$C_{1,2,1}(2n)$ and $C_{1,2,1}(2n+1)$.  For each subsequence, we see behavior that is
similar to that of $C_{0,1,l}(N)$.

 \begin{figure}[h] 
 \includegraphics[scale=1.00]{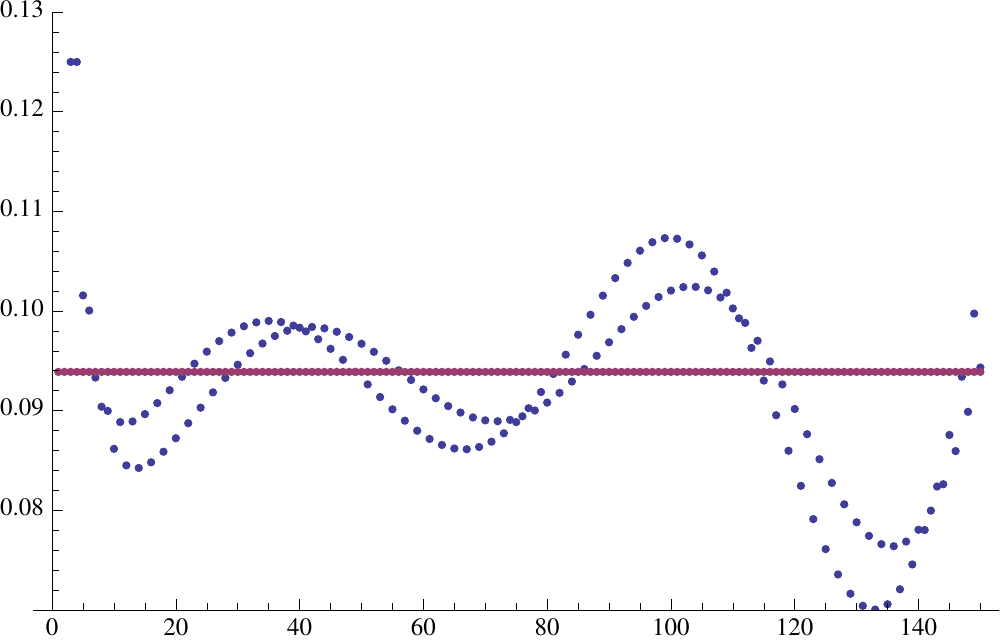}
 \caption{Graph of $C_{1,2,1}(N)$ for $N$ from $1$ to
 $150$, together with the line $y=R_{1,2,1}$ }
  \label{C121-150}
\end{figure}

 \begin{figure}[h!] 
 \includegraphics[scale=1.00]{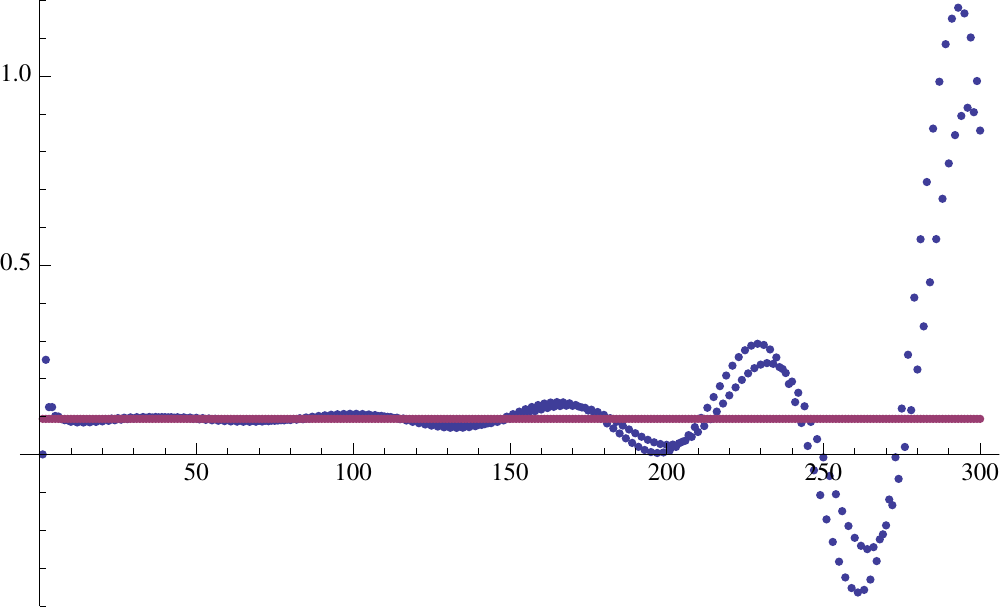}
 \caption{Graph of $C_{1,2,1}(N)$ for $N$ from $1$ to
 $300$, together with the line $y=R_{1,2,1}$ }
 \label{C121-300}
\end{figure}

\section{``Top down" formulas for the Rademacher Coefficients}
\subsection{$C_{0,1,l}(N)$}

As Rademacher already pointed out, it seems hopeless to get a closed-form formula for
$C_{h,k,l}(N)$ for $l=1,2, \dots$, {\it but} if you work your way down from the ``top'',
one can conjecture, and then {\it rigorously} prove explicit formulas for
$C_{h,k,N-r}(N)$, that alas, get increasingly more complicated as $r$ gets larger.

\begin{conj} \label{C01topdown}
 \[ C_{0,1,N-r}(N)  =  \frac{(-1)^{N+r}}{4^r N! r!} P_{0,1,N-r} (N), \]
 where, for $r>0$, $P_{0,1,N-r}(N)$ is a convex, alternating, monic polynomial of degree $2r$ whose
 only real roots are $0$ and $1$.
\end{conj}

\begin{theorem}  \label{P01topdown}Explicit formulas for the $P_{0,1,N-r}(N)$ of Conjecture~\ref{C01topdown} may
be given for any specific $r$,
in particular, we have
 \begin{equation} \label{P01N} P_{0,1,N}(N) = 1 \end{equation}
\begin{equation} \label{P01N-1}  P_{0,1,N-1}(N)= N^2 - N \end{equation}
\begin{equation} \label{P01N-2} P_{0,1,N-2}(N)= N^4-\frac{22 N^3}{9}+\frac{13 N^2}{3}-\frac{26 N}{9} \end{equation}
 \begin{equation} \label{P01N-3} P_{0,1,N-3}(N)=N^6-\frac{13 N^5}{3}+\frac{43 N^4}{3}-25 N^3+\frac{98 N^2}{3}-\frac{56 N}{3} \end{equation}
   \begin{multline}
   P_{0,1,N-4}(N) = N^8-\frac{20 N^7}{3}+\frac{862 N^6}{27}-\frac{21104 N^5}{225}+\frac{29039
   N^4}{135}-\frac{14548 N^3}{45} \\+\frac{9892 N^2}{27}-\frac{42896 N}{225}
   \end{multline}
 \end{theorem}
 
 \begin{remark}
 Readers desiring formulas for $C_{0,1,N-r}(N)$ for $r>4$ are directed to the
\texttt{ChkFormula} procedure in the
 \texttt{HANS}
Maple package.
 \end{remark}

\begin{remark}
From the above (and additional data not reproduced here but available at the website), we may deduce that
\begin{multline*}
   P_{0,1,N-r}(N) = N^{2r} - \frac{2r^2+7r}{9} N^{2r-1} + \frac{ 4r^4 + 12r^3 + 287r^2 - {303}r }{2\cdot 9^2} N^{2r-2} 
   \\- \frac{ 200r^6 - 300 r^5 + 40706 r^4 + 42939 r^3 - 257509 r^2 + 173964 r}{150\cdot 9^3} N^{2r-3} \\+ \mbox{ lower degree terms.} 
  \end{multline*}
   For $s>1$, the coefficients of $N^{2r-s}$ again appear to be polynomials in $r$ of degree $2s$, whose real roots include $0$ and $1$, although they are not convex and may have additional real roots.
\end{remark}

 \begin{proof}[Proof of Theorem~\ref{P01topdown}]
 Define $G_N := (x-1)^N F_N(x).$  Then $G_N$ has a Taylor series expansion about $x=1$, whose first $N$
 coefficients are the Rademacher coefficients:
 \[ G_N = \sum_{j=0}^{N-1} C_{0,1,N-j}(N)  (x-1)^j + \mbox{ higher degree terms }. \]
 Clearly, \begin{equation} \label{Grec} (1-x^N) G_N = (x-1) G_{N-1}. \end{equation}
 Expanding $(1-x^N)$ on the left hand side of~\eqref{Grec} as a Taylor polynomial about $x=1$, we have
 \begin{multline} \label{C011rec}
  \left( -\sum_{j=1}^N \binom{N}{j} (x-1)^j \right) \left( \sum_{j=0}^{N-1} C_{0,1,N-j}(N)  (x-1)^j + \mbox{ higher degree terms } 
  \right)
   \\ = \sum_{j=1}^{N-1} C_{0,1,N-j}(N-1) (x-1)^j + \mbox{higher degree terms}
 \end{multline}
 
Comparing the coefficients of $(x-1)^1$ on both sides of~\eqref{C011rec}, we find
\[ -N C_{0,1,N}(N) = C_{0,1,N-1}(N-1). \]  Solving the recurrence with the initial condition $C_{0,1,1}(1) = -1$,
yields 
\begin{equation} \label{C01NN} C_{0,1,N} (N) = \frac{(-1)^{N}}{N!}, \end{equation} which is~\eqref{P01N}.
 
Comparing the coefficients of $(x-1)^2$ on both sides of~\eqref{C011rec}, we find, taking into account~\eqref{C01NN},
\begin{equation}  -N C_{0,1,N-1}(N) - \binom N2 \frac{(-1)^N}{N!} =  C_{0,1,N-2}(N-1)  \end{equation}
with initial condition $C_{0,1,1}(2) = -\frac{1}{4}$ yields
\begin{equation} \label{C01N-1N} C_{0,1,N-1}(N) = \frac{(-1)^{N+1}}{4(N-2)!} , \end{equation}
which is~\eqref{P01N-1}.

Comparing the coefficients of $(x-1)^3$ on both sides of~\eqref{C011rec}, we find, taking into account~\eqref{C01NN}
and~\eqref{C01N-1N},
\begin{equation}  -N C_{0,1,N-2}(N) - \binom N2 \frac{(-1)^{N+1}}{4(N-2)!} 
-\binom N3 \frac{(-1)^N}{N!} =  C_{0,1,N-3}(N-1)  \end{equation}
with initial condition $C_{0,1,1}(3) = -\frac{17}{72}$ yields
\begin{equation} \label{C01N-2N} C_{0,1,N-2}(N) = \frac{(-1)^{N}(9N^2 - 13N + 26)}{288(N-2)!} , \end{equation}
which is~\eqref{P01N-2}.
Results for larger $r$ follow analogously.
\end{proof}

\subsection{$C_{1,2,l}(N)$}
Let us now define $\bar{G}_N$ analogously to $G_N$.
Let $$\bar{G}_N := (x+1)^{\lfloor N/2 \rfloor} F_N(x).$$  Then $\bar{G}_N$ has a Taylor series expansion about $x=-1$,
whose first $\lfloor \frac N2 \rfloor$ coefficients are the Rademacher coefficients:
\[ \bar{G}_N = \sum_{r=0}^{\lfloor N/2 \rfloor - 1} C_{1,2,\lfloor N/2 \rfloor - r}(N) (x+1)^r + \mbox{ higher degree terms } .\]
We now abandon the use of the floor function.
Notice that 
\begin{equation} 
   (1-x^{2n-1})(1-x^{2n}) \bar{G}_{2n} = (x+1)\bar{G}_{2n-2}.
\end{equation} 
Thus, by expanding the two left most factors on the left side as a Taylor series about $x=-1$,
\begin{multline} \label{C12even}
  \left\{ \sum_{r=1}^{4n-1} (-1)^{r+1} \left[ \binom{4n-1}{r} + \binom{2n}{r} - \binom{2n-1}{r} \right] (x+1)^r \right\} \\
  \times \left( \sum_{r=0}^{n-1} C_{1,2,n-r}(2n) (x+1)^r + \mbox{ higher degree terms }\right) \\
 = \sum_{r=1}^{n-1} C_{1,2,n-r} (2n-2) (x+1)^r + \mbox{ higher degree terms } 
\end{multline}

By comparing coefficients of $(x+1)^r$ in both sides of ~\eqref{C12even} 
and solving the recurrences, we obtain formulas
for $C_{1,2,n-r}(2n)$
analogous to those for $C_{0,1,N-r}(N)$.

 \begin{equation} \label{C12n2n}
    C_{1,2,n}(2n) = \frac{1}{2^{2n} n!}.
 \end{equation}

 \begin{equation} \label{C12n-12n}
    C_{1,2,n-1}(2n) = \frac{n}{2^{2n} (n-1)!}.
 \end{equation}

 \begin{equation} \label{C12n-22n}
    C_{1,2,n-2}(2n) = \frac{18n^3 - 8n^2 + 15n + 2}{9\cdot 2^{2n+2} (n-1)!}.
 \end{equation}

Of course, the observation 
$ (1-x^{2n})(1-x^{2n+1}) \bar{G}_{2n+1} = (x+1)\bar{G}_{2n-1} $
leads to analogous formulas for the $C_{1,2,n-r}(2n+1)$, e.g.,
\begin{equation}
 C_{1,2,n}(2n+1)= \frac{1}{2^{2n+1} \ n!},
\end{equation}
\begin{equation}
 C_{1,2,n-1}(2n+1)= \frac{2n^2+2n+1}{2^{2n+2}\ n!},
\end{equation}
\begin{equation}
C_{1,2,n-2}(2n+1) = \frac{18n^5+46n^4 + 61n^3 + 53n^2 + 29n + 9}{9\cdot 2^{2n+3}\  (n+1)!},
\end{equation}

Clearly, the same idea can be used to find formulas for $C_{h,k,n-j}(kn+r)$ for any $h$, $k$, $j$, $r$.
This has been implemented in the procedure \texttt{ChkFormula} in the \texttt{HANS} Maple package.
For those desiring automatically generated papers, containing both formulas of this type and their
proofs, please use the \texttt{HansTopDownAutoPaper} procedure in the \texttt{HANS} Maple package.

\section{Close Encounters of the Rademacher Kind}
While it appears that $\lim_{N\to\infty} C_{h,k,l}(N)$ does not exist for any $(h,k,l)$,  
we can nonetheless define  $B_{h,k,l}$ to be the $N$ which comes closest to 
$R_{h,k,l}$.
This is implemented in
the \texttt{CloseEncounters} procedure in \texttt{HANS}.

\begin{center}
\begin{tabular}{| c | c | r | r | } \hline
$l$ & $B_{0,1,l}$ & $\vert C_{0,1,l}( B_{0,1,l} ) - R_{0,1,l} \vert$ &$\vert C_{0,1,l}( B_{0,1,l} ) / 
R_{0,1,l} \vert$ \\
\hline
 $1$ &  $25$ & $0.0003177 $ & $0.99989 $\\
 $2$ &  $47$ & $0.0001434 $ & $0.99924 $\\
 $3$ &  $71$ & $0.0000828 $ & $0.99991 $\\
 $4$ &  $149$ & $0.0000009$ & $1.00001$ \\
 \hline
\end{tabular}
\end{center}

Notice that the first few values of $B_{0,1,l}$ are close to $24l$.  This motivates us to 
consider comparing $C_{0,1, l}(24l)$ to $R_{0,1,l}$.

\begin{center}
\begin{tabular}{| c | l | l | } \hline
$l$ & $\vert C_{0,1,l}(24l) - R_{0,1,l}\vert$ &  $\vert C_{0,1,l}(24l) / R_{0,1,l} \vert $ \\
\hline
 $1$ & $0.0053741095$ & $ 1.018346206$ \\
 $2$ & $0.0015044594$ & $ 1.007927400$ \\
 $3$ & $0.00033240887$ & $ 0.996241370$ \\
 $4$ & $0.00004427030$ & $ 1.001376635$ \\
 $5$ & $0.000011288321$ & 0.9988220859 \\
 $6$ & $0.000001686611$ &1.0006971253\\
 $7$ & $0.0000001275687$ & 0.9997575030\\
 $8$ & $0.0000000110523$ &    1.0000986383 \\
 $9$ & $0.00000000239242$ & 0.9999562770 \\
 $10$ & $0.000000005333208$ & 1.0000141594\\
 $11$ & $0.0000000187490584$ & 0.9999947242\\
 $12$ & $0.0000000393434274$ & 1.0000017401 \\ \hline
 \end{tabular}
 \end{center}
 
 Thus we have some evidence that even though the $N$ for which $C_{0,1,l}(N)$ is closest
 to $R_{0,1,l}$ is not $N=24l$, $C_{0,1,l}(24l)$ seems to provide a good approximation to
 $R_{0,1,l}$, and the approximation seems to be improving as $l$ increases.

\bibliographystyle{amsplain}

\end{document}